\documentclass{article}
\usepackage[utf8]{inputenc}
\usepackage{amsmath, amsthm, amsfonts, amssymb}
\usepackage{scalerel, geometry, datetime}
\usepackage{hyperref}

\newdateformat{daymonthyear}{%
  \THEDAY\ \monthname[\THEMONTH] \THEYEAR}

\renewcommand*{\S}{S_{\chi_1,\chi_2}}
\DeclareMathOperator{\smallmod}{mod}
\newcommand{\smallpmod}[1]{\ (\smallmod #1)}

\let\epsilon\varepsilon
\newcommand{\lowerparen}[2]{%
    \raisebox{-#1}{\(\displaystyle\left(\raisebox{#1}{\(\displaystyle #2\)}\right)\)}}
\newcommand{\sll}[1]{\!\! \ll_{#1} \!\!}
\newcommand{\sgg}[1]{\! \gg_{#1} \!}

\newtheorem{theorem}{Theorem}[section]
\newtheorem{lemma}[theorem]{Lemma}
\newtheorem{proposition}[theorem]{Proposition}
\newtheorem{corollary}[theorem]{Corollary}
\theoremstyle{definition}

\newtheorem*{remark}{Remark}

\title{An Average of Generalized Dedekind Sums}
\author{Travis Dillon\footnote{Lawrence University, \url{travisdillon01@gmail.com}} \ and Stephanie Gaston\footnote{California State University Dominguez Hills, \url{srgaston8294@gmail.com}}}
\date{\daymonthyear\today}

\begin{document}
\maketitle

\begin{abstract}
We study a generalization of the classical Dedekind sum that incorporates two Dirichlet characters and develop properties that generalize those of the classical Dedekind sum. By calculating the Fourier transform of this generalized Dedekind sum, we obtain an explicit formula for its second moment. Finally, we derive upper and lower bounds for the second moment with nearly identical orders of magnitude.
\end{abstract}

\section{Introduction}

Let $h$ and $k$ be coprime integers with $k>0$. The classical Dedekind sum is defined as
\[ s(h,k) = \sum_{j \smallmod k} B_1\!\left(\frac{j}{k}\right) B_1\!\left(\frac{jh}{k}\right), \]
where $B_1(x)$ is the first Bernoulli function
\[ B_1(x) = \begin{cases}
0 &\text{if } x\in\mathbb{Z}\\
x - \lfloor x \rfloor - \frac{1}{2} &\text{otherwise.} 
\end{cases} \]

Dedekind first studied this sum because of its connection to the transformation properties of the Dedekind eta function. Since then, Dedekind sums have been studied extensively in a wide variety of contexts, including modular forms, topology, combinatorial geometry, and quadratic reciprocity. 

Averages of Dedekind sums have attracted considerable attention. In \cite{walum-L-series}, Walum found an exact formula for the second moment of the classical Dedekind sum:
\begin{equation}\label{equation:walum-result}
 \sum_{a \smallmod p} \lvert s(a,p) \rvert ^2
 = \frac{p^2}{\pi^4(p-1)} \sum_{\substack{\psi \smallmod p \\ \psi(-1) = -1}} \lvert L(1,\psi)\rvert^4.
\end{equation}
Following this, Conrey, Fransen, Klein, and Scott \cite{conrey} and Zhang \cite{zhang-1996} studied the asymptotics of the second and higher moments.

Numerous authors have generalized the classical Dedekind sum in a variety of ways. In \cite{berndt1973}, Berndt develops a generalization of the classical Dedekind sum that incorporates a Dirichlet character and higher Bernoulli functions. Da\u{g}l\i{} and Can \cite{dagli-can-2015} generalized Berndt's sum to include two characters, and Stucker, Vennos, and Young \cite{last-year-dedekind} recently showed how to derive this generalized Dedekind sum using newform Eisenstein series. In this paper, we study the sum from \cite{dagli-can-2015} and \cite{last-year-dedekind}.  We follow the notation of \cite{last-year-dedekind}.

Let $\chi_1$ and $\chi_2$ be nontrivial primitive characters with moduli $q_1$ and $q_2$, respectively, such that $\chi_1\chi_2(-1) = 1$. For $\gamma = \left(\begin{smallmatrix} a & b \\ c & d\end{smallmatrix}\right) \in \Gamma_0(q_1q_2)$ with $c\geq 1$, the generalized Dedekind sum associated to $\chi_1$ and $\chi_2$ is given by
\begin{equation}\label{dedekind-sum-finite-formula}
\S(\gamma) = 
\sum_{j\smallmod c}\ \sum_{n \smallmod q_1}
\overline{\chi_2}(j)\overline{\chi_1}(n)
B_1\!\left(\frac{j}{c}\right) B_1\!\left(\frac{n}{q_1} + \frac{aj}{c}\right) \raisebox{-.2cm}{.}
\end{equation}
Because the generalized Dedekind sum depends only on the entries in the first column of $\gamma$, it will at times be convenient to write $\S(a,c)$ in place of $\S(\gamma)$. Additionally, note that the generalized Dedekind sum depends only on the residue of $a$ modulo $c$.

We will often make use of an alternate form of the generalized Dedekind sum. Let $e(x) = \exp(2\pi i x)$ and $\tau(\chi,l)$ denote the Gauss sum $\sum_{n \smallmod q} \chi(n)e(nl/q)$. We write $\tau(\chi)$ for $\tau(\chi,1)$. For any character $\chi$ modulo $q$, define the first generalized Bernoulli function as 
\begin{equation}\label{equation:B1chi-definition}
B_{1,\chi}(x) = \frac{-\tau(\overline{\chi})}{2\pi i}\sum_{l \neq 0} \frac{\chi(l)}{l}e\!\left(\frac{lx}{q}\right) \raisebox{-.2cm}{.}
\end{equation}
This equation unifies the cases appearing in Definition 1 of \cite{Euler-mac}. Theorem 3.1 of \cite{Euler-mac} states a finite sum formula for $B_{1,\chi}$ for primitive characters $\chi$:
\begin{equation}\label{equation:finite-sum-B1chi}
B_{1,\chi}(x) = \sum_{r \smallmod q} \overline{\chi}(r)B_1\!\left(\frac{x+r}{q}\right) \raisebox{-.2cm}{.}
\end{equation}
Substituting (\ref{equation:finite-sum-B1chi}) into (\ref{dedekind-sum-finite-formula}) gives the alternate formula 
\begin{equation}\label{equation:SwithB1chi}
\S(\gamma) =  \sum_{j \smallmod c} \overline{\chi_2}(j) B_1\!\left(\frac{j}{c}\right) B_{1,\chi_1}\!\left(\frac{aj}{c/q_1}\right) \raisebox{-.2cm}{.}
\end{equation}

Lemma 2.2 of \cite{last-year-dedekind} shows that the generalized Dedekind sum $\S(\gamma)$ is a crossed homomorphism from $\Gamma_0(q_1q_2)$ into $\mathbb{C}$. Set $\chi(\gamma) = \chi(d)$ for $\gamma = \left(\begin{smallmatrix} a & b \\ c & d \end{smallmatrix}\right)$, and let $\gamma_1,\gamma_2 \in \Gamma_0(q_1q_2)$. Then
\begin{equation}\label{equation:crossed-homomorphism}
\S(\gamma_1\gamma_2) = \S(\gamma_1) + \chi_1\overline{\chi_2}(\gamma_1)\S(\gamma_2).
\end{equation}
If $\chi_1 = \chi_2 = \chi$, then $S_{\chi,\chi}(\gamma)$ is a homomorphism from $\Gamma_0(q_1q_2)$ into $\mathbb{C}$.

Our main result is an exact formula for the second moment of $\S(\gamma)$, which generalizes the result of Walum in \cite{walum-L-series}.  We use $\chi^\star$ to denote the primitive character that induces a given character $\chi$ and $q(\chi)$ to denote its conductor.
\begin{theorem}\label{thm:second-moment}
Let $\chi_1$ and $\chi_2$ be nontrivial primitive characters modulo $q_1$ and $q_2$, respectively, such that $\chi_1\chi_2(-1) = 1$, and let $q_1q_2 \mid c$. Then
\begin{equation}\label{equation:second-moment}
\sum_{\substack{a \smallmod c \\ (a,c) = 1}} \lvert\S(a,c)\rvert^2 =
\frac{\varphi(c)}{\pi^4} \sum_{\substack{\psi \smallmod c \\ \psi\chi_1(-1)=-1}}
\lvert L(1,\overline{\psi}\vphantom{\psi}^\star\chi_1)\rvert^2 \lvert L(1,(\psi\chi_2)^\star)\rvert^2 
\vert g_{\chi_1,\chi_2}(\psi;c)\rvert^2,
\end{equation}
where
\begin{equation}\label{equation:definition-g}
g_{\chi_1,\chi_2}(\psi;c) = \tau((\overline{\psi\chi_2})^\star) \tau(\psi^\star) \tau(\overline{\chi_1})
\sum_{\substack{d \mid c \\ d \equiv 0\smallmod q(\psi)}}
    \frac{\overline{\chi_2}(c/d)}{\varphi(d)}
    ((\overline{\psi\chi_2})^\star\mu \ast 1)(d)\ 
    (\chi_1 \ast \mu\psi^\star)\!\left(\frac{d}{q(\psi)}\right) \raisebox{-.2cm}{.}
\end{equation}
\end{theorem}

We also bound the second moment above and below.
\begin{theorem}\label{thm:dedekind-sum-bounds}
Let $\chi_1$ and $\chi_2$ be nontrivial primitive characters modulo $q_1$ and $q_2$, respectively, such that $\chi_1\chi_2(-1) = 1$, and let $q_1q_2 \mid c$. Then
\[
\sum_{\substack{a \smallmod c \\ (a,c) = 1}} \lvert \S(a,c)\rvert^2 = q_1 c^{2+o(1)}.
\]
\end{theorem}
\noindent
Since (\ref{dedekind-sum-finite-formula}) has $cq_1$ terms, each of which is bounded by $1$ in absolute value, Theorem \ref{thm:dedekind-sum-bounds} is consistent with ``square-root'' cancellation.

The lower bound of Theorem \ref{thm:dedekind-sum-bounds} implies the following corollary, which may be interpreted as a statement about the kernel of the crossed homomorphism $\S(\gamma)$.
\begin{corollary}\label{thm:S-is-nonzero}
Let $\chi_1$ and $\chi_2$ be nontrivial primitive characters modulo $q_1$ and $q_2$, respectively, such that $\chi_1\chi_2(-1) = 1$. Then for each positive $c \equiv 0 \smallmod q_1q_2$, there exists an integer $a$ coprime to $c$ such that $S_{\chi_1,\chi_2}(a,c)$ is nonzero.
\end{corollary}
\noindent
In particular, the crossed homomorphism $\S(\gamma)$ is always nontrivial. In contrast, the only homomorphism from $\Gamma_0(1)=SL_2(\mathbb{Z})$ into $\mathbb{C}$ is trivial, since the abelianization of $SL_2(\mathbb{Z})$ is finite (for a nice proof of this well-known fact, see Keith Conrad's notes\footnote{\url{https://kconrad.math.uconn.edu/blurbs/grouptheory/SL(2,Z).pdf}}).

In Section \ref{section:properties}, we establish a few basic properties of the generalized Dedekind sum. These properties are extensions of corresponding properties of the classical Dedekind sum, but it appears that they have yet to be documented in the literature. One such property writes the generalized Dedekind sum as a cotangent sum: 
\begin{proposition}\label{thm:dedekind-sum-cotangent}
Let $\chi_1$ and $\chi_2$ be characters modulo $q_1$ and $q_2$, respectively. Then for each positive $c \equiv 0 \smallpmod{q_1q_2}$, 
\[ 
\S(\gamma)
= -\frac{1}{4cq_2} \sideset{}{'}\sum_{s \smallmod c }\ \, \sideset{}{'}\sum_{r \smallmod q_2}
    \cot\!\left(\frac{r\pi}{c}\right)\cot\!\left(\frac{s\pi}{c}\right)  
    \tau(\overline{\chi_1},s) \,\tau\left(\overline{\chi_2},\frac{(r+as)q_2}{c}\right), \]
where the primes indicate omission of the terms for which $\cot$ is undefined.
\end{proposition}
\noindent
Section \ref{section:properties} is independent of the following sections, as the proofs of Theorems \ref{thm:second-moment} and \ref{thm:dedekind-sum-bounds} do not depend upon its results.

Further work in this area could include sharper asymptotics of the second moment, along the lines of Zhang \cite{zhang-1996}, or asymptotics of higher moments, following Conrey, Fransen, Klein, and Scott \cite{conrey}.


\section{Properties of the generalized Dedekind sum}\label{section:properties}

\subsection{Arithmetic properties}
Throughout this section, let $\chi_1$ and $\chi_2$ be characters modulo $q_1$ and $q_2$, respectively. The original definition for the generalized Dedekind sum requires $\chi_1$ and $\chi_2$ to be nontrivial and primitive characters such that $\chi_1\chi_2(-1) = 1$. However, we can extend the definition to any pair of characters via the finite sum formula (\ref{dedekind-sum-finite-formula}). With this extension, we can recover the classical Dedekind sum by taking $\chi_1$ and $\chi_2$ to be trivial. Consequently, all of the properties proved in this section subsume the analogous properties for the classical Dedekind sum.

We may similarly extend the definition of the generalized Dedekind sum to all $a \in \mathbb{Z}$ and $c \equiv 0 \smallpmod{q_1q_2}$ through (\ref{dedekind-sum-finite-formula}). As mentioned in \cite{conrey}, it is not hard to prove that $s(h,k) = s(\alpha h, \alpha k)$ for all positive integers $\alpha$, so this extension adds no new features in the classical case. The same is true in the general case, due to the following:

\begin{proposition} \label{prop:dedekind-with-scaler}
Let $a$ and $c$ be coprime integers with $q_1q_2 \mid c$. Then $\S(a,c) = \S(\alpha a, \alpha c)$ for all positive integers $\alpha$.
\end{proposition}

\begin{proposition} \label{prop:dedekind-all-zero}
If $\chi_1\chi_2(-1) = -1$, then $\S(\gamma) = 0$ for all $\gamma \in \Gamma_0(q_1q_2)$.
\end{proposition}

We prove Proposition \ref{prop:dedekind-with-scaler} by rewriting the sum over $j$ in (\ref{dedekind-sum-finite-formula}) as
\[ \sum_{j \smallmod \alpha c} \longrightarrow \sum_{l \smallmod \alpha}\, \sum_{j \smallmod c} \]
and change variables $j \to j + c l$. Proposition \ref{prop:dedekind-all-zero} is proven by changing variables $j \to -j$ and $n\to -n$ in (\ref{dedekind-sum-finite-formula}) and using that $B_1(-x) = -B_1(x)$.

We also present two arithmetic properties of the generalized Dedekind sum. For the remainder of this section, we assume that $\chi_1\chi_2(-1) = 1$.

\begin{proposition}\label{prop:(-a,c)to(a,c)}
Let $c \geq 1$ and $q_1q_2 \mid c$. Then $\S(-a,c) = -\chi_2(-1) \S(a,c)$.
\end{proposition}

\begin{proposition}\label{prop:a-to-inverse}
Let $c \geq 1$ and $q_1q_2 \mid c$. Further, suppose that $a\overline{a} \equiv 1 \smallpmod{c}$. Then $S_{\chi_1,\chi_2}(\overline{a},c) = \chi_1(-a)\overline{\chi_2}(a)S_{\chi_1,\chi_2}(a,c)$.
\end{proposition}

Proposition \ref{prop:(-a,c)to(a,c)} can be proved by changing variables $j\to -j$ in the the finite sum formula (\ref{dedekind-sum-finite-formula}). To prove Proposition \ref{prop:a-to-inverse}, we substitute $j \rightarrow a(j-nc/q_1)$, then $n \to -\overline{a}n$ in (\ref{dedekind-sum-finite-formula}).  If $\gamma = \left(\begin{smallmatrix} a & b \\ c & d \end{smallmatrix}\right) \in \Gamma_0(q_1q_2)$, then $d \equiv \overline{a} \smallpmod{c}$, so Proposition \ref{prop:a-to-inverse} may be rewritten as $\S(d,c) = \chi_1(-1)\chi_2(d)\S(a,c)$.

The classical Dedekind sum can be written in a variety of forms. Using that
\[ \sum_{j \smallmod k} B_1\!\left(\frac{hj}{k}\right) = 0, \]
Apostol \cite[p. 61]{apostol-book} shows that 
\[ s(h,k) = \sum_{j = 1}^{k-1} \frac{j}{k}\, B_1\!\left(\frac{hj}{k}\right), \]
eliminating one of the Bernoulli functions. Analogously, substituting $j\to c-j$ and $n \to -n$ shows that
\[ \sum_{j \smallmod c}\ \sum_{n \smallmod q_1} \overline{\chi_2}(j)\overline{\chi_1}(n) B_1\left(\frac{n}{q_1} + \frac{aj}{c}\right) = 0. \]
Therefore, the generalized Dedekind sum can be written as
\begin{equation}
\S(\gamma) = 
\sum_{j=1}^{c-1}\ \sum_{n \smallmod q_1} \frac{j}{c}\, \overline{\chi_2}(j)\overline{\chi_1}(n) B_1\left(\frac{n}{q_1} + \frac{aj}{c}\right) \raisebox{-.2cm}{.}
\end{equation}

The classical Dedekind sum can also be written as a cotangent sum, as in equation 26 of \cite{dedekind-sums-book} and Exercise 11, Chapter 3 of \cite{apostol-book}:
\[ s(h,k) = \frac{1}{4k} \sum_{j = 1}^{k-1} \cot\!\left(\frac{\pi j}{k}\right) \cot\!\left(\frac{\pi hj}{k}\right) \raisebox{-.2cm}{.} \]
Proposition \ref{thm:dedekind-sum-cotangent} gives the corresponding result for the generalized Dedekind sum.

\subsection{Proof of Proposition \ref{thm:dedekind-sum-cotangent}}
Rademacher and Grosswald \cite[p.~14]{dedekind-sums-book} prove that
\[ B_1\!\left(\frac{j}{c}\right) = \sideset{}{'}\sum_{r \smallmod c} \left(\frac{ e(r/c) }{1-e(r/c)} + \frac{1}{2}\right)e\!\left(\frac{rj}{c}\right) \raisebox{-.2cm}{.} \]
With a bit of manipulation, this transforms into
\begin{equation}\label{equation:B1-as-cotangent}
B_1\!\left(\frac{j}{c}\right) = \frac{i}{2c}\ \sideset{}{'}\sum_{r \smallmod c} \cot\!\left(\frac{\pi r}{c}\right) e\!\left(\frac{rj}{c}\right) \raisebox{-.2cm}{.}
\end{equation}
Substituting (\ref{equation:B1-as-cotangent}) into (\ref{dedekind-sum-finite-formula}) gives
\begin{equation}\label{equation:cot-substitution}
\S(\gamma)
= -\frac{1}{4c^2}
\sideset{}{'}\sum_{r,s \smallmod c} \cot\!\left(\frac{r\pi}{c}\right) \cot\!\left(\frac{s\pi}{c}\right)
 \sum_{n \smallmod q_1} \overline{\chi_1}(n)  e\!\left(\frac{sn}{q_1} \right)
\sum_{j \smallmod c}\overline{\chi_2}(j)e\!\left(\frac{j(r+as)}{c}\right) \raisebox{-.2cm}{.}
\end{equation}
We substitute $j \to j + l q_2$ in the sum over $j$ to get
\begin{equation}\label{equation:cot-sum-over-j}
\sum_{j \smallmod c}\overline{\chi_2}(j)e\!\left(\frac{j(r+as)}{c}\right)
= \sum_{j \smallmod q_2}
    \overline{\chi_2}(j) e\!\left(\frac{j(r+as)}{c}\right)
    \sum_{l \smallmod c/q_2}
        e\!\left(\frac{l (r+as)}{c/q_2}\right) \raisebox{-.2cm}{.}
\end{equation}
From the othogonality relations for additive characters, the sum over $l$ is $c/q_2$ if $r+as \equiv 0$ (mod $c/q_2$) and vanishes otherwise. If $c/q_2$ divides $r+as$, then (\ref{equation:cot-sum-over-j}) simplifies as 
\begin{equation}\label{equation:cot-sum-over-l}
\frac{c}{q_2} \,\tau\left(\overline{\chi_2},\frac{(r+as)q_2}{c}\right) \raisebox{-.2cm}{.}
\end{equation}

The sum over $n$ in (\ref{equation:cot-substitution}) is the Gauss sum $\tau(\overline{\chi_1},s)$. Substituting the Gauss sums into (\ref{equation:cot-substitution}) gives
\[ 
\S(\gamma)
= -\frac{1}{4cq_2} \sideset{}{'}\sum_{\substack{r,s \smallmod c \\ r + as \equiv 0 \smallmod c/q_2}}
    \cot\!\left(\frac{r\pi}{c}\right)\cot\!\left(\frac{s\pi}{c}\right)  
    \tau(\overline{\chi_1},s) \,\tau\left(\overline{\chi_2},\frac{(r+as)q_2}{c}\right), \]
and changing variables $r\to -as+rc/q_2$ finishes the proof.  

A special case occurs when $\chi_1$ and $\chi_2$ are primitive. Then the Gauss sums simplify as $\tau(\overline{\chi},s) = \tau(\overline{\chi})\chi(s)$, and the formula in Proposition \ref{thm:dedekind-sum-cotangent} becomes
\[ \S(\gamma) = -\frac{\tau(\overline{\chi_1})\tau(\overline{\chi_2})}{4c q_2} \sideset{}{'}\sum_{s \smallmod c}\ \, \sideset{}{'}\sum_{ r \smallmod q_2} \chi_1(s) \chi_2(r) \cot\!\left(\pi\!\left(\frac{ r}{q_2}-\frac{as}{c}\right)\right) \cot\!\left(\frac{\pi s}{c}\right) \raisebox{-.2cm}{.} \]


\section{Proof of Theorem \ref{thm:second-moment}}

We prove Theorem \ref{thm:second-moment} by evaluating the Fourier transform of $\S$. Our approach will loosely follow that of Walum in \cite{walum-L-series}. Before we begin, we establish a character relation (Lemma \ref{thm:character-sum-congruence}) and evaluate three finite sums (Lemmas \ref{thm:sum-characters_B1}, \ref{thm:sum-character_B1chi}, and \ref{thm:sum-character_rel_prime}). 

Throughout this section, $\chi_1$ and $\chi_2$ are nontrivial primitive characters modulo $q_1$ and $q_2$, respectively, such that $\chi_1\chi_2(-1) = 1$.  We denote the principal character modulo $q$ by $\chi_{0,q}$.
Finally, we extend every Dirichlet character $\psi$ to all of $\mathbb{R}$ by setting $\psi(x)=0$ if $x\not\in\mathbb{Z}$.

\subsection{Preliminary lemmas}

\begin{lemma}\label{thm:character-sum-congruence}
For any $c>0$,
\[
\sum_{d \mid c} \frac{1}{\varphi(d)}\sum_{\psi \smallmod d} \psi\left(\frac{m}{c/d}\right)\overline{\psi}\left(\frac{n}{c/d}\right) = \begin{cases} 1 &\text{if } m \equiv n \pmod c\\
0 &\text{if } m\not\equiv n \pmod{c}.
\end{cases}
\]
\end{lemma}
\begin{proof}
Suppose that $m \equiv n$ (mod $c$) and $(m,c) = k$. The sum over $\psi$ is $0$ unless $c/d = k$, in which case it is $\varphi(d)$. On the other hand, if $m\not\equiv n$ (mod $c$), then the inner sum always vanishes.
\end{proof}

\begin{lemma}\label{thm:sum-characters_B1}
If $\chi$ is a character modulo $d \mid c$, then
\[ \sum_{r \smallmod c} \chi(r) B_1\!\left(\frac{r}{c}\right) = -\frac{\tau(\chi^\star)}{\pi i} (\chi^\star\mu \ast 1)(d)
\begin{cases}
L(1,\overline{\chi}^\star) &\text{if } \chi(-1)=-1\\
0 &\text{if } \chi(-1)=1.
\end{cases}
\]
\end{lemma}
\begin{proof}
Substituting
\[
\chi(r) = \chi^\star(r) \sum_{k \mid (d,r)} \mu(k),
\]
into the sum over $r$ to get
\begin{equation}\label{equation:sum-B1}
\sum\limits_{r \smallmod c} \chi(r) B_1\left(\frac{r}{c}\right)
= \sum_{k \mid d} \mu(k) \chi^\star(k) \sum_{r \smallmod c/k}  \chi^\star(r)B_1\left(\frac{r}{c/k}\right) \raisebox{-.2cm}{.}
\end{equation}
Consider the inner sum. If $(q(\chi),k) > 1$, then $\chi^\star(k)=0$. We may therefore assume that $(q(\chi),k)=1$, which implies that $q(\chi) \mid c/k$. We substitute $r \to t + rq(\chi)$, where $t$ runs modulo $q(\chi)$ and $r$ now runs modulo $c/kq(\chi)$. Using the orthogonality relations, the sum over $r$ in (\ref{equation:sum-B1}) becomes
\[
\sum_{r = 0 }^{c/kq(\chi)-1}\ \sum_{t = 1}^{q(\chi)-1} \chi^\star(t) \left(\frac{t+rq(\chi)}{c/k} - \frac{1}{2}\right) 
= \sum_{t = 1}^{q(\chi)-1} \chi^\star(t) \sum_{r = 0 }^{c/kq(\chi)-1} \frac{t}{c/k}.
\]
The inner sum simplifies to $t/q(\chi)$. Using orthogonality and (\ref{equation:finite-sum-B1chi}), we obtain
\[
\sum_{t = 1}^{q(\chi)-1} \chi^\star(t) \frac{t}{q(\chi)}
= \sum_{t\smallmod q(\chi)} \chi^\star(t) B_1\!\left(\frac{t}{q(\chi)}\right)
= B_{1,\overline{\chi}^\star}(0).
\]
Applying (\ref{equation:B1chi-definition}) gives
\begin{equation}\label{equation:chiB1-last-eq}
B_{1,\overline{\chi}^\star}(0)
= -\frac{\tau(\chi^\star)}{\pi i}\begin{cases} L(1,\overline{\chi}^\star) &\text{if } \chi(-1)=-1\\
0 &\text{if } \chi(-1)=1.
\end{cases}
\end{equation}
Substituting (\ref{equation:chiB1-last-eq}) for the sum over $r$ in (\ref{equation:sum-B1}) proves the lemma.
\end{proof}

The following result is Lemma 3.2 of \cite{kowalski-book} (corrected on Kowalski's website).
\begin{lemma}\label{thm:gauss-nonprimitve}
If $\psi$ is a non-principal character modulo $d$ then 
\[ \tau(\psi,l) =
\tau(\psi^\star) \sum_{\substack{k \mid l \\ k\mid d/q(\psi)}}
    k\, \overline{\psi}\vphantom{\psi}^\star\!\!\left(\frac{l}{k}\right)
    \psi^\star\!\!\left(\frac{d}{kq(\psi)}\right) \mu\!\left(\frac{d}{kq(\psi)}\right) \raisebox{-.2cm}{.}
\]
\end{lemma}

\begin{lemma}\label{thm:sum-character_B1chi}
If $\psi$ is a character modulo $d \mid c$ and $\chi$ is a primitive character modulo $q\mid c$, then
\[
\sum_{t \smallmod c}
    \psi\!\left(\frac{t}{c/d}\right) B_{1,\chi}\!\left(\frac{t}{c/q}\right)
= -\frac{\tau(\overline{\chi}) \tau(\psi^\star)}{\pi i}\,
    (\chi \ast \mu\psi^\star)\!\left(\frac{d}{q(\psi)}\right)
    \begin{cases}
        L(1,\chi\overline{\psi}\vphantom{\psi}^\star)  &\text{if } \chi\psi(-1) = -1\\
        0                               &\text{if } \chi\psi(-1) =  1.
    \end{cases}
\]
\end{lemma}
\begin{proof}
If $c/d \nmid t$, then the term vanishes, so we substitute $t \to (c/d)t$ and express $B_{1,\chi}$ as an infinite sum using (\ref{equation:B1chi-definition}):
\begin{align} 
\sum_{t \smallmod c} \psi\left(\frac{t}{c/d}\right) B_{1,\chi}\left(\frac{t}{c/q}\right)
&= \sum_{t \smallmod d} \psi(t)B_{1,\chi}\!\left(\frac{t(c/d)}{c/q}\right) \nonumber\\
&= -\frac{\tau(\overline{\chi})}{2\pi i}\sum_{l\not=0} \frac{\chi(l)}{l} \sum_{t \smallmod d} \psi(a)\, e\!\left(\frac{tl}{d}\right) \raisebox{-.2cm}{.} \label{equation:sum-lemma-B1chi-infinite}
\end{align}

The sum over $t$ is the Gauss sum $\tau(\psi,l)$. Using Lemma \ref{thm:gauss-nonprimitve}, (\ref{equation:sum-lemma-B1chi-infinite}) becomes
\[
-\frac{\tau(\overline{\chi}) \tau(\psi^\star)}{2\pi i}
    \sum_{k \mid d/q(\psi)} \psi^\star\!\left(\frac{d}{kq(\psi)}\right)
    \mu\!\left(\frac{d}{kq(\psi)}\right) \chi(k)
    \sum_{l \not= 0} \frac{\chi(l)}{l} \overline{\psi}\vphantom{\psi}^\star(l).
\]
The sum over $l$ is $2L(1,\chi\overline{\psi}\vphantom{\psi}^\star)$ if $\chi\psi(-1) = -1$ and $0$ otherwise, which completes the proof.
\end{proof}

\begin{lemma}\label{thm:sum-character_rel_prime}
If $\chi$ is a character modulo $q \mid c$, then \[ \sum_{\substack{a \smallmod c \\ (a,c) = 1}} \chi(a) = \begin{cases} \varphi(c) & \text{if $\chi$ is principal} \\
0 & \text{otherwise.}
\end{cases} \] 
\end{lemma}

\begin{proof}
The sum is equal to
\[ 
\sum_{a \smallmod c } \chi(a)\chi_{0,c}(a)
= \begin{cases}
\varphi(c) &\text{if } \chi\chi_{0,c} \text{ is principal}\\
0 &\text{otherwise,}
\end{cases} \]
and $\chi\chi_{0,c}$ is principal exactly when $\chi$ is principal.
\end{proof}

In the next theorem, we fix $c$ and regard $\S(a,c)$ as a function of $a$ from $(\mathbb{Z}/c\mathbb{Z})^\times$ into $\mathbb{C}$.

\begin{subsection}{Fourier transform of $\S$}
\end{subsection}
\begin{theorem}\label{thm:fourierTransform}
Let $\xi$ be a character modulo $c$ where $q_1q_2 \mid c$. The finite Fourier transform of $\S(a,c)$ is
\[ \widehat{S}_{\chi_1,\chi_2}(\xi) = 
\sum_{\substack{a \smallmod c \\ (a,c) = 1}} \S(a,c) \xi(a)=
\frac{\varphi(c)}{(\pi i)^2}\begin{cases}
L(1,(\xi\chi_2)^\star) L(1,\overline{\xi}\vphantom{\xi}^\star\chi_1)
g_{\chi_1,\chi_2}(\xi;c) & \text{if } \xi\chi_1(-1) = -1\\
0 &\text{if } \xi\chi_1(-1) = 1,
\end{cases}\]
where $g_{\chi_1,\chi_2}(\xi;c)$ is as defined in (\ref{equation:definition-g}).
\end{theorem}
\begin{proof}
Substituting with (\ref{equation:SwithB1chi}) into $\widehat{S}_{\chi_1,\chi_2}(\xi)$ and setting $t \equiv ar$ (mod $c$), we have
\begin{equation}\label{equation:expanded_moment}
\widehat{S}_{\chi_1,\chi_2}(\xi) =
\sum_{\substack{a \smallmod c \\ (a,c)=1}}\ \,
\sum_{\substack{r,t \smallmod c \\ t \equiv ar \smallmod c}} \overline{\chi_2}(r)B_1\!\left(\frac{r}{c}\right)B_{1,\chi_1}\!\left(\frac{t}{c/q_1}\right)\xi(a).
\end{equation}

We use Lemma \ref{thm:character-sum-congruence} to replace the condition $t \equiv ar$ in (\ref{equation:expanded_moment}).  Because $(a,c)=1$, we have $\psi(\frac{ar}{c/d})=\psi(a)\psi(\frac{r}{c/d})$. This substitution separates the variables:
\begin{multline} \label{equation:expanded_moment-seperated-variables}
\sum_{d \mid c}\frac{1}{\varphi(d)}
\sum_{\psi \smallmod d}\,
\Bigg(\,\sum_{\substack{a \smallmod c \\ (a,c) = 1}}\overline{\psi}\xi(a)\Bigg)
\left(\sum_{r \smallmod c} \overline{\chi_2}(r)\overline{\psi}\left(\frac{r}{c/d}\right)B_1\!\left(\frac{r}{c}\right) \right) \cdot \\ 
\left(\sum_{t \smallmod c} \psi\left(\frac{t}{c/d}\right)B_{1,\chi_1}\!\left(\frac{t}{c/q_1}\right) \right) \raisebox{-.3cm}{.}
\end{multline}
Consider the sum over $a$.  The character $\overline{\psi}\xi$ is principal exactly when $\psi^\star = \xi^\star$. Using Lemma \ref{thm:sum-character_rel_prime}, (\ref{equation:expanded_moment-seperated-variables}) becomes
\[
\varphi(c)\sum_{d \mid c}\frac{1}{\varphi(d)}
\sum_{\substack{\psi \smallmod d \\ \psi^\star = \xi^\star}}\,
\left(\sum_{r \smallmod c} \overline{\chi_2}(r)\overline{\psi}\left(\frac{r}{c/d}\right)B_1\!\left(\frac{r}{c}\right) \right)
\left(\sum_{t \smallmod c} \psi\left(\frac{t}{c/d}\right)B_{1,\chi_1}\!\left(\frac{t}{c/q_1}\right) \right) \raisebox{-.3cm}{.}
\]

If $q(\xi) \mid d$, then there is exactly one character $\psi$ modulo $d$ such that $\psi^\star = \xi^\star$, namely $\xi^\star\chi_{0,d}$. If $q(\xi) \nmid d$, then there is no such character. 

Therefore, we may change the double sum over $d$ and $\psi$ to a single sum over $d$ as shown
\[ \sum_{d \mid c} \sum_{\substack{\psi \smallmod d \\ \psi^\star = \xi^\star}} \ \longrightarrow
 \sum_{\substack{d \mid c \\ d \equiv 0\smallmod q(\xi)}}, \]
where the character $\psi$ corresponding to $d$ in the right-hand sum is $\xi^\star\chi_{0,d}$. 

The sum over $r$ is
\begin{equation}\label{equation:fourier-transform-sum-r}
\sum_{r \smallmod c} \overline{\chi_2}(r)\overline{\psi}\!\left(\frac{r}{c/d}\right)B_1\!\left(\frac{r}{c}\right)
= \overline{\chi_2}(c/d) \sum_{r \smallmod d} \overline{\chi_2}(r)\overline{\psi}(r)B_1\!\left(\frac{r}{d}\right) \raisebox{-.2cm}{.}
\end{equation}
We may assume that $q_2$ and $c/d$ are coprime. Therefore $q_2 \mid d$, which implies that $\chi_2\psi$ is a character modulo $d$. Applying Lemmas \ref{thm:sum-characters_B1} and \ref{thm:sum-character_B1chi} to the sum over $r$ in (\ref{equation:fourier-transform-sum-r}) and the sum over $t$ in (\ref{equation:expanded_moment-seperated-variables}), respectively, completes the proof.
\end{proof}

To prove Theorem \ref{thm:second-moment}, apply Parseval's formula
\[ \sum_{\substack{a \smallmod c \\ (a,c) = 1}} \lvert\S(a,c)\rvert^2 =
\frac{1}{\varphi(c)}
\sum_{\psi \smallmod c}
    \Bigg\lvert \sum_{\substack{a \smallmod c \\ (a,c) = 1}} \S(a,c) \psi(a) \Bigg\rvert^2 \raisebox{-.3cm}{.} \]
The result follows immediately from substitution using Theorem \ref{thm:fourierTransform}.


\section{Proof of Theorem \ref{thm:dedekind-sum-bounds}}

We prove Theorem \ref{thm:dedekind-sum-bounds} by obtaining the upper bound and lower bounds individually. In this section, $\chi_1$ and $\chi_2$ are nontrivial primitive characters modulo $q_1$ and $q_2$, respectively, such that $\chi_1\chi_2(-1) = 1$, and $q_1q_2 \mid c$.

\subsection{Upper bound}
\begin{theorem}\label{thm:second-moment-upper-bound}
For every $\epsilon > 0$, we have
\[ \sum_{\substack{a \smallmod c \\ (a,c) = 1}} \lvert\S(a,c)\rvert^2
\ll_\epsilon
q_1 c^{2+\epsilon}. \]
\end{theorem}
\begin{proof}
It is well known (see \cite[chapter 14]{davenport-book}) that there exists some $K>0$ so that for any character $\chi$ modulo $q$, we have
\[ \lvert L(1,\chi)\rvert \leq K \log q. \]
From applying this bound to (\ref{equation:second-moment}), it follows that
\begin{equation} \label{equation:upperBound2moment}
\sum_{\substack{a \smallmod c \\ (a,c) = 1}} \lvert\S(a,c)\rvert^2
\leq
K^4 \log(c)^4 \varphi(c) \sum_{\substack{\psi \smallmod c \\ \psi\chi_1(-1)=-1}}
\vert g_{\chi_1,\chi_2}(\psi;c)\rvert^2.
\end{equation}
We now bound the absolute value of $g$ as defined in (\ref{equation:definition-g}). Using the triangle inequality and that $q(\psi\chi_2) \leq c$, we have
\begin{equation}\label{equation:upper-bound-first-eq}
\lvert g_{\chi_1\chi_2}(\psi;c)\rvert^2
\leq q_1 c q(\psi)\lowerparen{3pt}{
    \sum_{\substack{d \mid c \\ d \equiv 0\smallmod q(\psi)}}
    \frac{1}{\varphi(d)}
    \Big\lvert ((\overline{\psi\chi_2})^\star\mu \ast 1)(d)\Big\rvert
    \left\lvert (\chi_1 \ast \mu\psi^\star)\!\left(\frac{d}{q(\psi)}\right)\right\rvert }^2 \raisebox{-.4cm}{.}
\end{equation}
The convolutions are bounded by the divisor function:
\[
\Big\lvert ((\overline{\psi\chi_2})^\star\mu \ast 1)(d) \Big\rvert 
\leq \sum_{k \mid d} \left\lvert \mu(k)\, (\psi \chi_2)^\star(k) \right\rvert\\
\leq \sigma_0(d)
\leq \sigma_0(c),
\]
where $\sigma_k(n) = \sum_{d\mid n} d^k$ is the sum of divisors function. Likewise,
\[
\left\lvert (\chi_1 \ast \mu\psi^\star)\!\left(\frac{d}{q(\psi)}\right) \right\rvert
\leq \sum_{k\mid d/q(\psi)} \left\lvert \chi_1\!\left(\frac{d}{k q(\psi)}\right) \mu(k)\psi^\star(k) \right\rvert\\
\leq \sigma_0\!\left(\frac{d}{q(\psi)}\right)
\leq \sigma_0(c).
\]
For each $\epsilon > 0$, we have the bound $\sigma_0(c) \sll{\epsilon}  c^\epsilon$ for all positive integers $c$. Similarly, we have $\varphi(d) \sgg{\epsilon} d^{1-\epsilon}$.

Substituting these bounds into (\ref{equation:upper-bound-first-eq}) then using that $d \geq q(\psi)$, we find that the terms inside the sum over $d$ are all at most $c^{\epsilon}/q(\psi)$. The sum contains at most $\sigma_0(c)$ terms, so, in all,
\begin{equation}\label{equation:upper-bound-g-epsilons}
\lvert g_{\chi_1,\chi_2}(\psi;c)\rvert^2
\ll_\epsilon q_1 \frac{c^{1+\epsilon}}{q(\psi)} .
\end{equation}

Inserting (\ref{equation:upper-bound-g-epsilons}) into (\ref{equation:upperBound2moment}), gives
\begin{equation}\label{equation:upper-bound-S-last}
\sum_{\substack{a \smallmod c \\ (a,c) = 1}} \lvert\S(a,c)\rvert^2
\ll_\epsilon
q_1 \varphi(c) \log(c)^4 \sum_{\psi \smallmod c}\frac{c^{1+\epsilon}}{q(\psi)}
\ll_\epsilon
q_1 c^{2+\epsilon} \sum_{\psi \smallmod c}\frac{1}{q(\psi)}.
\end{equation}
We reorder the sum over $\psi$ by its possible conductors. The number of characters modulo $c$ with conductor $d$ is at most $\varphi(d)$, so
\begin{equation}\label{equation:sum-reciprocal-conductor}
\sum_{\psi \smallmod c}\frac{1}{q(\psi)} \leq \sum_{d\mid c} \frac{\varphi(d)}{d} \leq \sigma_0(c) \ll_\epsilon  c^\epsilon.
\end{equation}
Inserting (\ref{equation:sum-reciprocal-conductor}) into (\ref{equation:upper-bound-S-last}) completes the proof.
\end{proof}

\subsection{Lower bound}
We first prove some lemmas on primitive characters.  Let $\varphi^\star(n)$ denote the number of primitive characters modulo $n$. For $n\geq 1$ and $p$ a prime, we have $\varphi^\star(p^n) = \varphi(p^n) - \varphi(p^{n-1})$. Explicitly,
\begin{equation}\label{equation:psi-star}
\varphi^\star(p^n) = \begin{cases}
    p^n \left(1-\frac{1}{p}\right)^2    &\text{if } n > 1\\
    p - 2                               &\text{if } n = 1.
\end{cases}    
\end{equation}

\begin{lemma}\label{thm:number-odd-primitive}
There are at least $\varphi^\star(p^{n})/2 - 1$ primitive characters with a given parity modulo $p^{n}$.
\end{lemma}
\begin{proof} If we let $\sum^\star$ denote summation over primitive characters, then the sum
\[ \sideset{}{^\star}\sum_{\chi \smallmod p^n} \frac{1-\chi(-1)}{2} \]
counts the number of odd primitive characters modulo $p^n$. To finish the proof, we distribute the sum and use equation (3.8) of \cite{kowalski-book}, which states that
\begin{equation}\label{equation:prim-char-sum-lemma}
\sideset{}{^\star}\sum_{\chi \smallmod q} \chi(m) =
\begin{cases}
0   &\text{if } (m,q) > 1\\
\sum\limits_{\substack{d \mid m-1 \\ d \mid q}} \mu(d) \varphi\left(\frac{q}{d}\right)   &\text{otherwise.}
\end{cases}
\end{equation}
To count the number even primitive characters, we use the sum
\[ \sideset{}{^\star}\sum_{\chi \smallmod p^n} \frac{1+\chi(-1)}{2} \]
and apply (\ref{equation:prim-char-sum-lemma}) again.
\end{proof}

\begin{lemma}\label{thm:primitive-cond-count}
Let $p$ be a prime, $\xi$ be a character modulo $p^k$, and $n\geq k$. For each of the following three conditions and any $\epsilon > 0$, there exists $K_\epsilon > 0$ so that the number of primitive characters $\psi$ modulo $p^n$ with a given parity and which satisfy that condition is bounded below by $K_\epsilon p^{n(1-\epsilon)}$:
\begin{enumerate}
    \item $\psi\xi$ is primitive, where $n > k$.
    \item $\psi\xi$ is primitive, where $p > 3$.
    \item $q(\psi\xi) \geq p^{n-1}$ for $p^n \not= 3, 4, 8$.
\end{enumerate}
\end{lemma}
\begin{proof}
We first prove the desired bound for all three conditions when $n > k$. In this case, $\psi\xi$ is primitive for every primitive character $\psi$ modulo $p^n$. So Lemma \ref{thm:number-odd-primitive} shows that there are $\varphi^\star(p^n)/2 - 1$ characters that satisfy the chosen condition. Inspection of (\ref{equation:psi-star}) proves the lemma in this case.

Now suppose that $n = k$. To combine the proofs for parts 2 and 3, let $m \in \{0,1\}$. We count the number of primitive characters of the given parity that do not satisfy the condition. If $q(\psi\xi) < p^{n-m}$, then $\psi = \psi'\overline{\xi}$ for some character $\psi'$ modulo $p^{n-m-1}$. There are $\varphi(p^{n-m-1})$ such characters. We separate the cases $n-m > 1$ and $n-m = 1$.

If $n-m > 1$, exactly half of the characters modulo $p^{n-m-1}$ characters make the product $\psi\xi$ have the desired parity. Subtracting these characters from the total number of primitive characters modulo $p^n$ with the same parity, we find that there are at least $\frac{1}{2}(\varphi^\star(p^n) - \varphi(p^{n-m-1})) - 1$ suitable primitive characters.

A straightforward calculation reveals that for $n-m > 1$,
\[
\varphi^\star(p^n) - \varphi(p^{n-m-1}) = p^n\left(1-\frac{1}{p}\right)^2 - p^{n-m-1}\left(1-\frac{1}{p}\right).
\]
Inspection of the case $m=0$ proves part 2, and inspection of the case $m=1$ proves part 3 with the exception of the case $p^n = 16$. By directly considering the group of Dirichlet characters modulo $16$, we see that there is at least one primitive character that satisfies the conditions of part 3.

Now suppose that $n-m = 1$, so that $p^{n-m-1} = 1$. There is at most one character $\psi'$ so that $\psi\xi$ has the desired property, so there are $\frac{1}{2}\phi^\star(p^n) - 2$ suitable primitive characters. As before, inspection proves parts $2$ and $3$ for $p^n \not= 16$. The remaining case is proved in the same way as before.
\end{proof}

\begin{remark}
Part 3 of Lemma \ref{thm:primitive-cond-count} holds for $p^n = 8$ in the case that $\psi\xi$ is odd.
\end{remark}

\begin{theorem}
For every $\epsilon > 0$, we have
\[ \sum_{\substack{ a \smallmod c \\ (a,c) = 1}} \lvert \S(a,c)\rvert^2
\gg_\epsilon q_1 c^{2-\epsilon} \raisebox{-.1cm}{.} \]
\end{theorem}
\begin{proof}
If $\chi$ is a real character modulo $q \mid c$, then Siegel's lower bound (see \cite[Theorem 11.14]{montgomery-book}) gives $\lvert L(1,\chi) \rvert \sgg{\epsilon} q^{-\epsilon}$. If $\chi$ is complex, then we have the stronger result $\lvert L(1,\chi) \rvert \sgg (\log q)^{-1}$ (see \cite[Theorem 11.4]{montgomery-book}). In either case, $\lvert L(1,\chi) \rvert \sgg{\epsilon} c^{-\epsilon}$.

We bound the second moment below by restricting the sum over $\psi$ in (\ref{equation:second-moment}) to primitive characters:
\begin{equation}\label{equation:lower-bound-S}
\sum_{\substack{ a \smallmod c \\ (a,c) = 1}} \lvert \S(a,c)\rvert^2
\gg_\epsilon
c^{1-\epsilon} \sideset{}{^\star}\sum_{\substack{\psi \smallmod c \\ \psi\chi(-1) = -1}} \lvert g_{\chi_1,\chi_2}(\psi;c)\rvert^2 \raisebox{-.1cm}{.}
\end{equation}
Since $\psi$ is primitive, $q(\psi) = c$ and the sum over $d$ in (\ref{equation:definition-g}) has only the term where $d=c$:
\begin{equation}\label{equation:lower-bound-g}
\sideset{}{^\star}\sum_{\substack{\psi \smallmod c \\ \psi\chi_1(-1) = -1}} \lvert g_{\chi_1,\chi_2}(\psi;c)\rvert^2
= \frac{cq_1}{\varphi(c)^2}\sideset{}{^\star}\sum_{\substack{\psi \smallmod c \\ \psi\chi_1(-1) = -1}}
    q(\psi\chi_2) \left\lvert \sum_{k \mid c} \mu(k)\,(\psi\chi_2)^\star(k) \right\rvert^2 \raisebox{-.4cm}{.}
\end{equation}

Heuristically, we would like to pick $\psi$ modulo $c$ so that $\psi\chi_2$ is primitive modulo $c$. Then $q(\psi\chi_2) = c$ and $k$ must be $1$, leading to a direct lower bound. There are, however, a few problematic cases where no such $\psi$ exists. Since there is only one primitive character modulo $3$ and modulo $4$, we cannot pick such a $\psi$ when $3 \| c$ or $4\| c$. We separate these problematic cases by factoring $c$.

Suppose that $c$ has factorization $2^{\alpha_2}3^{\alpha_3}c'$ and $q_2$ has a factorization $2^{\beta_2}3^{\beta_3}q_2'$, where $c'$ and $q_2'$ are coprime to $6$. For ease of notation, set $c_2 = 2^{\alpha_2}$ and $c_3 = 3^{\alpha_3}$. Write $\chi_2$ and $\psi$ as products of primitive characters $\xi_2\xi_3\xi'$ and $\psi_2\psi_3\psi'$ corresponding to these factorizations.

For the time being, suppose that $c_2 > 4$, $c_3 > 3$, and $c' > 1$. The following method will require slight modification in the complementary cases. By positivity, the sum over $\psi$ in (\ref{equation:lower-bound-g}) is bounded below by
\begin{equation}\label{equation:lower-bound-restrictions}
\sideset{}{^\star}\sum_{\substack{\psi_2 \smallmod c_2 \\ q(\psi_2\xi_2) \geq c_2/2 \\ \psi_2\xi_2(-1) = -1}}\hspace{1em}
\sideset{}{^\star}\sum_{\substack{\psi_3 \smallmod c_3 \\ q(\psi_3\xi_3) \geq c_3/3 \\ \psi_3\xi_3(-1)=-1}}\hspace{1em}
\sideset{}{^\star}\sum_{\substack{\psi' \smallmod c' \\ \psi'\xi' \text{ is primitive}\\ \psi'\xi'(-1)=-1}}
q(\psi\chi_2) \left\lvert \sum_{k \mid c} \mu(k)(\psi\chi_2)^\star(k) \right\rvert^2.
\end{equation}
If we set
$ h(\psi,\xi; k,l) = q(\psi\xi)\, (\psi\xi)^\star(k)\,(\overline{\psi\xi})^\star(l), $
then (\ref{equation:lower-bound-restrictions}) factors as
\begin{multline}\label{equation:expanded-lower-bound-g}
\sum_{\substack{ k_2, l_2 \mid c_2 \\ k_3, l_3 \mid c_3 \\ k', l' \mid c'}}
\mu(k_2k_3k')\mu(l_2l_3l') 
\lowerparen{7pt}{\ \sideset{}{^\star}\sum_{\substack{\psi_2 \smallmod c_2 \\ q(\psi_2\xi_2) \geq c_2/2  \\ \psi_2\xi_2(-1)=-1}}
    h(\psi_2,\xi_2;k_2k_3k',l_2l_3l')} \cdot \\ 
\lowerparen{7pt}{\ \sideset{}{^\star}\sum_{\substack{\psi_3 \smallmod c_3 \\ q(\psi_3\xi_3) \geq c_3/3 \\ \psi_3\xi_3(-1)=-1}} 
    h(\psi_3, \xi_3; k_2k_3k', l_2l_3l')}
\lowerparen{7pt}{\sideset{}{^\star}\sum_{\substack{\psi' \smallmod c' \\ \psi'\xi' \text{ is primitive}\\ \psi'\xi'(-1)=-1}}
    h(\psi', \xi'; k_2k_3k', l_2l_3l')} \raisebox{-.5cm}{.}
\end{multline}

Using Lemma \ref{thm:primitive-cond-count}, the assumption that $c_2 > 4$, $c_3 > 3$, and $c' > 1$ guarantees that each of these sums is nonempty. Since $c_2 > 4$, we know that $q(\psi_2\xi_2)$ is even.  But $k_2$ and $l_2$ are both powers of $2$, so each term in the sum over $\psi_2$ in (\ref{equation:expanded-lower-bound-g}) vanishes unless $k_2 = l_2 = 1$.  Similar considerations for the other two sums yield that $k_3 = l_3 = 1$ and $k' = l' = 1$.

Using that $h(\psi,\xi;1,1) = q(\psi\xi)$ and the restrictions $q(\psi_2\xi_2)\geq c_2/2$ and $q(\psi_3\xi_3)\geq c_3/3$, we bound (\ref{equation:expanded-lower-bound-g}) below by
\begin{equation}\label{equation:sum-counting-characters}
\frac{c_2c_3c'}{6}
\lowerparen{7pt}{\,\sideset{}{^\star}\sum_{\substack{\psi_2 \smallmod c_2 \\ q(\psi_2\xi_2) \geq c_2/2 \\ \psi_2\xi_2(-1)=-1}}
    1 }
\lowerparen{7pt}{\,\sideset{}{^\star}\sum_{\substack{\psi_3 \smallmod c_3 \\ q(\psi_3\xi_3) \geq c_3/3 \\ \psi_3\xi_3(-1)=-1}}
    1 }
\lowerparen{7pt}{\sideset{}{^\star}\sum_{\substack{\psi' \smallmod c' \\ \psi'\xi' \text{ is primitive}\\ \psi\chi_1(-1)=-1}}
    1 } \raisebox{-.5cm}{.}
\end{equation}
Parts 2 and 3 of Lemma \ref{thm:primitive-cond-count} show that the product of sums is bounded below by $c^{1-\epsilon}$, so (\ref{equation:sum-counting-characters}) is bounded below by $q_1 c^{2-\epsilon}$. Combining this bound with (\ref{equation:lower-bound-S}) finishes the proof for this case.

The proof of the remaining cases proceed similarly after modifying the parity conditions on the sums in (\ref{equation:lower-bound-restrictions}). Recall that $q_1q_2 \mid c$. If $c' = 1$, then either $c_2 > 2^{\beta_2}$ or $c_3 > 3^{\beta_3}$. Change the restriction on the sum over $\psi_3$ to $\psi_3\xi_3(-1)=1$. The same steps as before yield that $k_2=l_2=k_3=l_3=1$. Now the bound $q_1 c^{2-\epsilon}$ follows from parts 1 and 3 of Lemma \ref{thm:primitive-cond-count}.

If $c_2 \leq 4$, then either $\psi_2 = \xi_2$ or $4 \nmid q_2$. Suppose that $\psi_2 = \xi_2$. Factor the sum over $\psi$ in (\ref{equation:lower-bound-S}) as a sum over $\psi_2$, $\psi_3$, and $\psi'$. The sum over $\psi_2$ is bounded below by $c_2/4$, so (\ref{equation:lower-bound-S}) is again bounded below by $q_1 c^{2-\epsilon}$. On the other hand, if $4 \nmid q_2$, then factor $c$ as $c_3c'$ and use similar arguments to the ones employed above. The case $c_3 \leq 3$ is nearly identical.
\end{proof}


\noindent
{\large\textbf{Acknowledgements}} \\
This research was conducted at the 2019 REU hosted at Texas A\&M University and supported by NSF grant DMS-1757872. The authors thank Wei-Lun Tsai for his Sage advice and Professor Matthew Young for his support and guidance throughout the project.

\bibliography{REU2019-bibliography}
\bibliographystyle{abbrv}

\end{document}